\documentclass[12pt]{amsart}
\usepackage{amsfonts}
\usepackage{latexsym}
\usepackage{amssymb}
\usepackage{amsmath}
\usepackage{color}
\usepackage[pdftex]{graphicx}
\usepackage{bbm}
\usepackage{tikz}
\usepackage{enumerate}
\usepackage{mathrsfs}
\usepackage{todonotes}

\usepackage[normalem]{ulem} 

\usepackage[left=2.5cm, top=2.5cm, bottom=2.5cm, right=2.5cm]{geometry}

\newcommand{\R}{\mathbb R}
\newcommand{\N}{\mathbb N}

\newcommand{\E}{\mathbb E}
\newcommand{\Pro}{\mathbb P}

\newcommand{\Var}{\mathrm{Var}}

\def\dint{\textup{d}}
\newcommand{\SSS}{\ensuremath{{\mathbb S}}}

\newcommand{\B}{\ensuremath{{\mathbb B}}}

\newcommand{\eps}{\varepsilon}
\newcommand{\eqdistr}{\stackrel{d}{=}}
\newcommand{\todistr}{\overset{d}{\underset{n\to\infty}\longrightarrow}}
\newcommand{\todistrnoN}{\overset{d}{\underset{}\longrightarrow}}

\newtheorem{thm}{Theorem}[section]

\newtheorem{lemma}[thm]{Lemma}
\newtheorem{df}[thm]{Definition}
\newtheorem{proposition}[thm]{Proposition}

{
\theoremstyle{definition}

\newtheorem{rmk}[thm]{Remark}

}

\def\bX{\mathbf{X}}
\def\bY{\mathbf{Y}}

\usepackage{nomencl}
\makenomenclature

\allowdisplaybreaks

\begin{document}


\title[]{Yet another note on the \\ arithmetic-geometric mean inequality}

\author[Z. Kabluchko]{Zakhar Kabluchko}
\address{Zakhar Kabluchko: Institut f\"ur Mathematische Stochastik, Westf\"alische Wilhelms-Uni\-ver\-sit\"at M\"unster, Germany}
\email{zakhar.kabluchko@uni-muenster.de}

\author[J. Prochno]{Joscha Prochno}
\address{Joscha Prochno: Institut f\"ur Mathematik \& Wissenschaftliches Rechnen, Karl-Franzens-Universit\"at Graz, Austria} \email{joscha.prochno@uni-graz.at}

\author[V. Vysotsky]{Vladislav Vysotsky}
\address{Vladislav Vysotsky: Department of Mathematics, University of Sussex, United Kingdom, and St.\ Petersburg Department of Steklov Mathematical Institute} \email{v.vysotskiy@sussex.ac.uk}

\keywords{arithmetic-geometric mean inequality; central limit theorem; large deviations; large deviations principle}
\subjclass[2010]{Primary: 52A23, 60F05, 60F10 Secondary: 46B06, 46B07}

\thanks{J. Prochno is supported by a Visiting Professor Fellowship of the Ruhr University Bochum and its Research School PLUS as well as by the Austrian Science Fund (FWF) Project F5508-N26, which is part of the Special Research Program ``Quasi-Monte Carlo Methods: Theory and Applications''. Part of this work was done while J. Prochno visited V. Vysotsky at the University of Sussex and Z. Kabluchko at the University of M\"unster. We thank the departments for their financial support.}


\begin{abstract}
It was shown by E.\ Gluskin and V.D.\ Milman in [GAFA Lecture Notes in Math.\ 1807, 2003] that the classical arithmetic-geometric mean inequality can be reversed (up to a  multiplicative constant) with high probability, when applied to coordinates of a point chosen with respect to the surface unit measure on a high-dimensional Euclidean sphere. We present here two asymptotic refinements of this phenomenon in the more general setting of the surface probability measure on a high-dimensional $\ell_p$-sphere, and also show that sampling the point according to either the cone probability measure on $\ell_p$ or the uniform distribution on the ball enclosed by $\ell_p$ yields the same results. First, we prove a central limit theorem, which allows us to identify the precise constants in the reverse inequality. Second, we prove the large deviations counterpart to the central limit theorem, thereby describing the asymptotic behavior beyond the Gaussian scale, and identify the rate function. 
\end{abstract}

\maketitle

\tableofcontents

\section{Introduction and main results}

The classical inequality of arithmetic and geometric means states that the arithmetic mean of a finite sequence of non-negative real numbers is greater than or equal to the geometric mean of the sequence, i.e., for any $n\in\N$ and $x_1,\dots,x_n\in[0,\infty)$,
\[
\bigg(\prod_{i=1}^{n}x_i\bigg)^{1/n} \leq \frac{1}{n}\sum_{i=1}^n x_i
\]
with equality if and only if $x_1=\dots=x_n$. This inequality may be written in the form
\[
\bigg(\prod_{i=1}^{n}y_i\bigg)^{1/n} \leq \sqrt{\frac{1}{n}\sum_{i=1}^n y_i^2},
\]
where $y_1,\dots,y_n>0$. This means that if $y=(y_1,\dots,y_n)$ is an element of the Euclidean unit sphere $\SSS^{n-1}$, then
\[
\bigg(\prod_{i=1}^{n}|y_i|\bigg)^{1/n} \leq \frac{1}{\sqrt{n}}.
\]
A natural question is how sharp this inequality is for a typical point on the sphere. In \cite[Proposition 1]{GM2003}, Gluskin and Milman showed that for large $n\in\N$ the arithmetic and geometric means are actually equivalent (up to multiplicative constants) with very high probability. More precisely, if we denote by $\sigma$ the rotationally invariant surface probability measure on $\SSS^{n-1}$, then for any $t\in(0,\infty)$,
\begin{align}\label{ineq: gluskin milman}
\sigma\left(\bigg\{x\in\SSS^{n-1}\,:\, \bigg(\prod_{i=1}^{n}|x_i|\bigg)^{1/n} \geq t \cdot \frac{1}{\sqrt{n}}\bigg\} \right) \geq 1-\big(1.6\sqrt{t}\,\big)^n.
\end{align}
In other words, if we sample a point uniformly at random on the unit Euclidean sphere, then it will satisfy a reverse (up to constant) arithmetic-geometric mean inequality with high probability. Alternatively, we can sample a point uniformly at random from the unit Euclidean ball, and put it to the sphere dividing by its norm.

The problem has been revisited by Aldaz in \cite[Theorem 2.8]{A2010}, who showed that for every $k,\varepsilon>0$ there exists an $N=N(k,\varepsilon)\in\N$ such that for every $n\geq N$,
\begin{align}\label{ineq:concentration aldaz}
\sigma\left(\bigg\{x\in\SSS^{n-1}\,:\, \frac{(1-\varepsilon)e^{-\frac{1}{2}(\gamma + \log 2)}}{\sqrt{n}}< \bigg(\prod_{i=1}^{n}|x_i|\bigg)^{1/n} < \frac{(1+\varepsilon) e^{-\frac{1}{2}(\gamma + \log 2)}}{\sqrt{n}}\bigg\} \right) \geq  1-\frac{1}{n^k},
\end{align}
where $\gamma$ denotes Euler's constant. This identifies the exact constant $e^{-\frac{1}{2}(\gamma + \log 2)}$ around which the ratio of geometric and arithmetic means concentrates. Aldaz also obtained a similar result for points chosen on the $\ell_1^n$-sphere (with concentration around the constant $e^{-\gamma}$) and studied weighted versions of the arithmetic-geometric mean inequality. For a refinement of the arithmetic-geometric mean inequality, we refer the reader to \cite{A2008}. 
\vskip 2mm

In this note we complement the inequalities \eqref{ineq: gluskin milman} and \eqref{ineq:concentration aldaz} by finding the (logarithmic) asymptotics of their left-hand sides. Our first observation is the following central limit theorem. We state and prove it in the setting of the inequality between geometric and $p$-generalized means, which says that for all $p >0$, $n\in\N$, and $x_1,\dots,x_n\in \R$,
\[
\bigg(\prod_{i=1}^{n}|x_i|\bigg)^{1/n} \leq \bigg(\frac{1}{n}\sum_{i=1}^n |x_i|^p\bigg)^{1/p}.
\]
Since both sides of this inequality scale linearly, we can assume that the right-hand side equals (or does not exceed) one. If $1\leq p <\infty$, this means that  $(x_1,\ldots,x_n)$ belongs to  a unit $\ell_p^n$-sphere (or an $\ell_p^n$-ball), for which we use the respective standard notation
\[
\SSS_p^{n-1} = \{x\in\R^n \,:\, \|x\|_p=1 \} \quad\text{and}\quad \B_p^n = \big\{x\in\R^n:\|x\|_p\leq 1\big\} 
\]
with the $\|\cdot\|_p$-norm of $x=(x_1,\dots,x_n)\in\R^n$ given by
\[
\|x\|_p = \bigg(\sum_{i=1}^n |x_i|^p\bigg)^{1/p}.
\]
Then it is natural to study the behavior of the geometric mean for typical $x_i$'s, which corresponds to choosing $x$ at random. There are several natural probability measures on $\B_p^n$ and $\SSS_p^{n-1}$. Restricting the Lebesgue measure to $\B_p^n$ and normalizing it, we obtain the uniform probability distribution on $\B_p^n$.
We shall also consider the surface probability measure and cone probability measure on $\SSS_p^{n-1}$, which we denote respectively by $\sigma_p$ and $\mu_p$; see Subsection \ref{subsec:lp balls} for the definition of $\mu_p$.

Recall that for $x>0$, the digamma function $\psi$ is defined via
\[
\psi(x) := \frac{\Gamma'(x)}{\Gamma(x)}\,,
\]
where $\Gamma$ denotes the gamma function and $\Gamma'$ its derivative.

\begin{thm} \label{thm:CLT}
Let $n\in\N$ and $p\in[1,\infty)$. Suppose that $\bX_n=(X_1^{(n)}, \ldots, X_n^{(n)})$ is a random vector that is either uniformly distributed over $\B_p^n$ or distributed according to $\mu_p$ or $\sigma_p$ on $\SSS_p^{n-1}$. Then, for every $a\in\R$,
$$
\lim_{n \to \infty} \Pro\bigg[ \Big(\prod_{i=1}^{n}|X_i^{(n)}|\Big)^{1/n} \geq e^{m_p} \Big(1 + \frac a {\sqrt{n}}\Big) \cdot \Big(\frac1n \sum\limits_{i=1}^n {|X_i^{(n)}|}^p\Big)^{1/p} \bigg] = 1 - \Phi \left( \frac{pa}{\sqrt{\psi'(\frac1p)-p}} \right),
$$
where $\Phi$ denotes the distribution function of a standard normal random variable and
\[
m_p:=\frac{\psi(\frac{1}{p}) + \log p }{p}
\]
is a negative constant only depending on $p$.
\end{thm}

In other words, the sequence of (random) ratios of geometric and $p$-generalized means, given by
\begin{equation}\label{eq:ratio agm}
\mathcal R_n := \frac{\Big(\prod\limits_{i=1}^{n}|X_i^{(n)}|\Big)^{1/n}}{ \Big(\frac1n \sum\limits_{i=1}^n {|X_i^{(n)}|}^p\Big)^{1/p}}\,,\quad n\in\N,
\end{equation}
satisfies a central limit theorem with the normalization $\sqrt n (e^{-m_p} \mathcal R_n -1)$. 

\begin{rmk}
Some particular values of $e^{m_p}$ for $p\in\{1,2,4\}$ and $p\to+\infty$ are
\[
e^{m_p} =
\begin{cases}
 e^{-\gamma} \approx 0.561 & :\, p=1 \\
 \exp \Big(-\frac{\gamma+\log 2}{2}\Big) \approx 0.529 & :\, p=2 \\
 \exp \Big(-\frac{2 \gamma +\pi + 2 \log 2}{8} \Big) \approx 0.491 & :\, p= 4\\
e^{-1} & :\, p \to \infty ,
\end{cases}
\]
and $\psi'(1)=\frac{\pi^2}{6}$, $\psi'(\frac12)=\frac{\pi^2}{2}$ (see \cite[Section~8.366]{GradshteynRyzhik}). Let us also note that, setting $a=0$ in Theorem \ref{thm:CLT}, we obtain
$$
\lim_{n \to \infty} \Pro\bigg[ \Big(\prod_{i=1}^{n}|X_i^{(n)}|\Big)^{1/n} \leq e^{m_p} \cdot \Big(\frac1n \sum\limits_{i=1}^n {|X_i^{(n)}|}^p\Big)^{1/p} \bigg] = \frac{1}{2},
$$
which means that with probability approaching $1/2$ the inequality between geometric and $p$-generalized means can be improved with the multiplicative constant $e^{m_p}<1$. Similarly, using Theorem \ref{thm:CLT} with $a\to +\infty$ and $a\to -\infty$, we see that  with probability approaching $1$, the inequality holds true with  any constant $c>e^{m_p}$ but ceases to hold with any constant $c<e^{m_p}$. The precise rate of convergence of these probabilities will be identified in our large deviations result, Theorem~\ref{thm:ldp}.
\end{rmk}

\begin{rmk}\label{rem: reduced form}
As will be shown at the end of the proof of Theorem \ref{thm:CLT}, we also have
\[
\lim_{n \to \infty} \Pro\bigg[ \Big(\prod_{i=1}^{n}|X_i^{(n)}|\Big)^{1/n} \geq e^{m_p} \Big(1 + \frac a {\sqrt{n}}\Big) \cdot n^{-1/p} \bigg] = 1 - \Phi \left( \frac{pa}{\sqrt{\psi'(\frac1p)-p}} \right),
\]
which is of course trivially follows from the theorem if the distribution of $\bX_n$ is $\mu_p$ or $\sigma_p$.
\end{rmk}


Our second observation concerns large deviations of the ratio $\mathcal R_n$ given in \eqref{eq:ratio agm}. While large deviations are extensively studied in probability theory (see, e.g., \cite{DZ,dH} and the references cited therein), they have not been considered -- contrary to central limit theorems --  in geometric functional analysis until the very recent paper by Gantert, Kim, and Ramanan \cite{GKR}. Already shortly after, this work has been extended and complemented in \cite{APT2018, KPT17CCM, KPT2018, Kim2017, KimRamanan2015}. In contrast to the universality in central limit theorems, the probabilities of  (large) deviations on the scale of laws of large numbers, are non-universal, thus being sensitive to the distribution of the  random variables considered. This non-universality is reflected by the so-called rate function, which essentially defines the large deviations probabilities.

\begin{thm}\label{thm:ldp}
Let $n \in \N$ and $p\in[1,\infty)$. Suppose that $\bX_n=(X_1^{(n)}, \ldots, X_n^{(n)})$ is a random vector that is either uniformly distributed over $\B_p^n$ or distributed according to $\mu_p$ or $\sigma_p$ on $\SSS_p^{n-1}$. Then
$$
\lim_{n \to \infty} \frac1n \log \Pro\bigg[ \Big(\prod_{i=1}^{n}|X_i^{(n)}|\Big)^{1/n} \geq \theta \cdot \Big(\frac1n \sum\limits_{i=1}^n {|X_i^{(n)}|}^p\Big)^{1/p} \bigg] = -\mathcal J_p(\theta), \qquad \theta \in [e^{m_p},1),
$$
and
$$
\lim_{n \to \infty} \frac1n \log \Pro\bigg[ \Big(\prod_{i=1}^{n}|X_i^{(n)}|\Big)^{1/n} \leq \theta \cdot \Big(\frac1n \sum\limits_{i=1}^n {|X_i^{(n)}|}^p\Big)^{1/p} \bigg] = -\mathcal J_p(\theta), \qquad \theta \in (0,e^{m_p}],
$$
where for $\theta \in (0,1)$,
\[
\mathcal J_p(\theta)  := [p G_p(\theta) -1] \log(\theta) + G_p(\theta) \big[\log\big(G_p(\theta)\big) - 1\big] -\log \Gamma\big(G_p(\theta)\big)  + \frac{1}{p} + \frac{1}{p}\log(p) + \log \Gamma\left(\frac 1p\right)
\]
with $G_p(\theta) := H^{-1}\big(p\log(\theta)\big)$, where $H:(0,\infty) \to (-\infty,0)$ is an increasing bijection given by
\[
H(x) := \psi(x)-\log(x).
\]
The function $\mathcal J_p$ is non-negative,
satisfies $\mathcal J_p(e^{m_p})=0$, and $\mathcal J_p(0+) = \mathcal J_p(1-) = +\infty.$
\end{thm}

\begin{center}
\includegraphics[height=6cm]{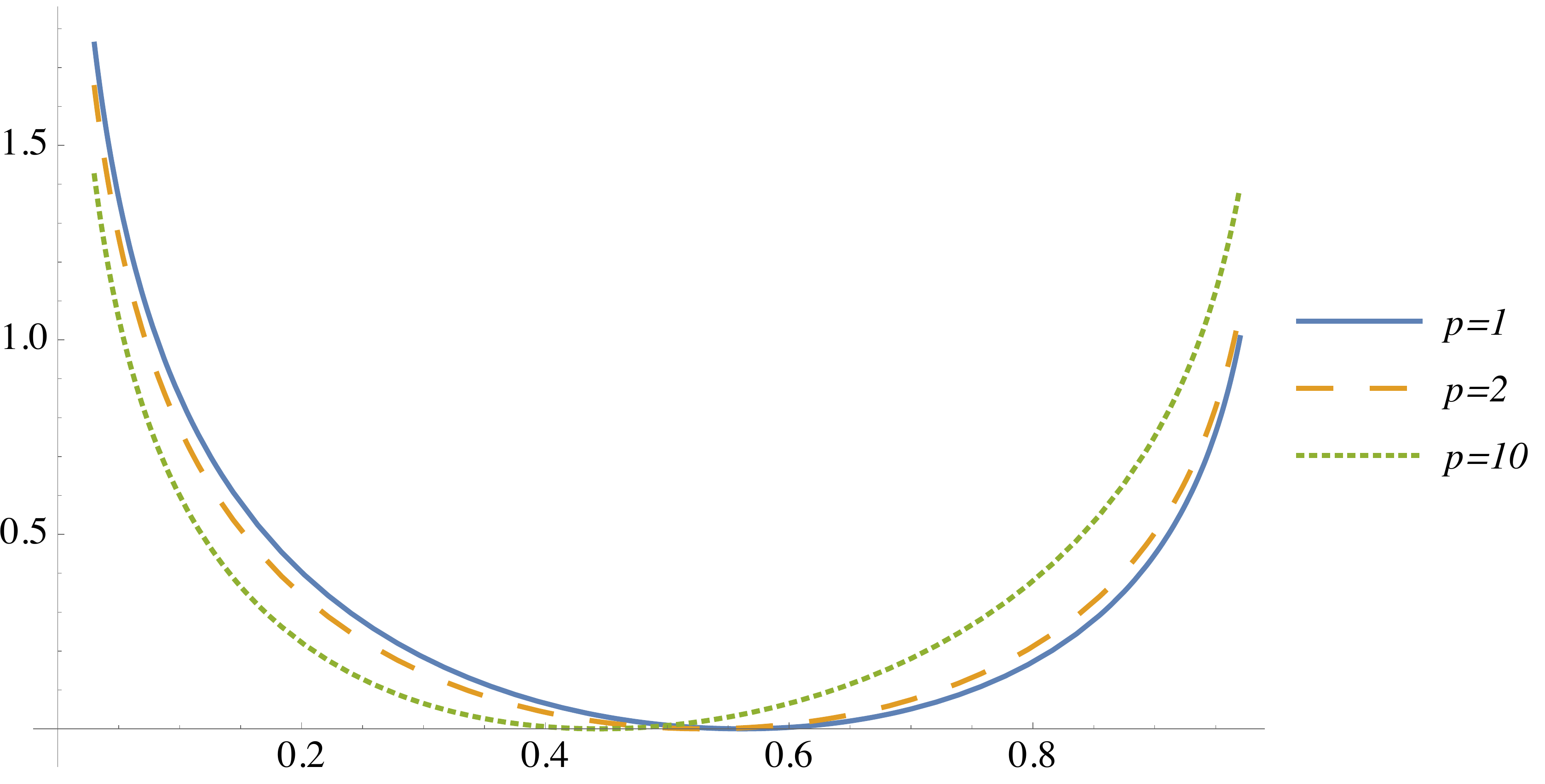}
\label{fig}

Figure~\ref{fig}. The rate function $\mathcal J_p$ for $p=1,2,10$.
\end{center}	

\begin{rmk} \label{rem}
(a) Note that the equality $\mathcal J_p(e^{m_p})=0$ agrees with Theorem~\ref{thm:CLT}. This equality also follows directly from the identity $m_p=\frac1p H(\frac1p)$.

\noindent (b) In fact, we shall prove the following so-called large deviations principle for $\mathcal R_n$ with the rate function $\mathcal J_p$. Put $\mathcal J_p \equiv +\infty$ on $\R \backslash (0,1)$, then for all Borel measurable sets $A\subseteq \R$,
\begin{equation*}
-\inf_{\theta\in A^\circ}\mathcal J_p(\theta) \leq\liminf_{n\to\infty}{1\over n}\log \Pro[\mathcal R_n\in A]
\leq\limsup_{n\to\infty}{1\over n}\log \Pro[\mathcal R_n \in A] \leq -\inf_{\theta\in\overline{A}}\mathcal J_p(\theta),
\end{equation*}
where $A^\circ$ and $\overline A$ refer to the interior and closure of $A$, respectively.

\noindent (c) In the setting of the uniform distribution on the sphere, Theorems \ref{thm:CLT} and \ref{thm:ldp} use results for the Radon--Nikodym density of $\mu_p$ and $\sigma_p$ (see Proposition \ref{prop:relation cone surface measure} and Lemma \ref{lem:estimate density}). Their extension to the regime $0<p<1$ are not fully known (see, e.g., \cite[Section 5, Comment (5)]{NR2003}), so for the uniform distribution on the sphere we remain in the regime $1\leq p<+\infty$. For the other two distributions one can easily see that the results continue to hold in the regime $0<p<1$ by following verbatim the proofs presented.
\end{rmk}

\section{Preliminaries}

We shall present here the notation and background material used throughout the text. We split this into appropriate subsections. Having a broad readership from both probability theory and geometric functional analysis in mind, we introduce the material on large deviations in slightly more detail.

\subsection{Notation}

We denote by $\R^n$ the $n$-dimensional Euclidean space and equip this with its standard inner product for which we write $\langle\cdot,\cdot \rangle$. For a subset $A$ of $\R^n$, we denote by $A^\circ$ the interior of $A$ and by $\bar{A}$ its closure. For a Borel measurable set $A\subseteq \R^n$, we shall denote by $|A|$ its $n$-dimensional Lebesgue measure. We shall also use the asymptotic notation $\sim$ to denote that the ratio of functions or sequences tends to $1$.

\subsection{The $\ell_p^n$-balls}\label{subsec:lp balls}

Let us recall some material regarding the geometry and probability of $\ell_p^n$-balls. For any $p\in[1,+\infty)$ the $\ell_p^n$-norm of $x=(x_1,\ldots,x_n)\in\R^n$ is given by
\[
\|x\|_p :=
\Big(\sum\limits_{i=1}^n|x_i|^p\Big)^{1/p}.
\]
For any $n$ and $p$ let us denote by $\B_p^n:=\{x\in\R^n:\|x\|_p\leq 1\}$ the unit ball and by $\SSS_p^{n-1}:=\{x\in\R^n:\|x\|_p=1\}$ unit sphere with respect to this norm. The restriction of the Lebesgue measure to $\B_p^n$ provides a natural volume measure on $\B_p^n$. We will supply $\SSS_p^{n-1}$ with the cone probability measure $\mu_p$ defined as follows: for a Borel set $A\subseteq \SSS_p^{n-1}$,
\begin{align}\label{defn:cone}
\mu_p(A) := \frac{|\{rx:x\in A,r\in[0,1]\}|}{|\B_p^n|}\,.
\end{align}
Let $\sigma_p$ be the $(n-1)$-dimensional Hausdorff probability measure or, equivalently, the $(n-1)$-dimensional normalized Riemannian volume measure on $\SSS_p^{n-1}$, $p \in[1,+\infty)$. We remark that the cone measure $\mu_p$ coincides with $\sigma_p$ if and only if $p=1$ and $p=2$. In particular, $\mu_2$ is the same as the normalized spherical Lebesgue measure. We shall use the following result on the form of the Radon--Nikodym density of cone and surface measure, which was proved in \cite[Lemma 2]{NR2003}.
\begin{proposition}\label{prop:relation cone surface measure}
Let $n\in\N$ and $1\leq p <+\infty$. Then, for all $x=(x_1,\dots,x_n)\in\SSS^{n-1}_p$,
\[
h_{n,p}(x):= \frac{\dint \sigma_p}{\dint \mu_p}(x) = C_{n,p}\cdot \Big(\sum_{i=1}^n|x_i|^{2p-2}\Big)^{1/2},
\]
where
\[
C_{n,p} := \bigg(\int_{\SSS_p^{n-1}} \Big(\sum_{i=1}^n|x_i|^{2p-2}\Big)^{1/2} \, \mu_p(\dint x) \bigg)^{-1}.
\]
\end{proposition}

We refer to \cite{NaorTAMS,NR2003} for more details on the relation between these two measures. We shall also use the following result. 

\begin{lemma}\label{lem:estimate density}
Let $1\leq p < +\infty$. There is a constant $C=C(p)\in(0,+\infty)$ such that, for all $n\in\N$ and every $x\in\SSS_p^{n-1}$,
\[
n^{-C} \leq h_{n,p}(x) \leq n^{C}.
\]
\end{lemma}
\begin{proof}
It suffices to show that there is a constant $A>0$ such that $n^{-2A}\leq \sum_{i=1}^n|x_i|^{2p-2}
\leq n^{2A}$ for all $x\in\SSS^{n-1}_p$. Indeed, from this it would follow that $n^{-A} \leq C_{n,p}\leq n^{A}$ by the definition of $C_{n,p}$, which would yield the claim. Write $y_i:= |x_i|^p\geq 0$, so that $\sum_{i=1}^n y_i = 1$. It is an easy consequence of H\"older's inequality that under this constraint we always have $\sum_{i=1}^{n} y_i^{\alpha} \leq \max\{n^{1-\alpha},1\}$ and $\sum_{i=1}^{n} y_i^{\alpha} \geq \min\{n^{1-\alpha},1\}$ for all $\alpha\geq 0$.  Taking $\alpha =(2p-2)/p$, we obtain the required bounds on $\sum_{i=1}^{n} y_i^{\alpha} = \sum_{i=1}^n|x_i|^{2p-2}$.
\end{proof}

The proofs of our results rely on the following probabilistic representation for the cone probability measure on $\SSS_p^{n-1}$ for $p\in[1,\infty)$ (and for the uniform distribution over $\B_p^n$), which is due to Schechtman and Zinn \cite{SchechtmanZinn} and was independently obtained by Rachev and R\"uschendorf in \cite{RachevRueschendorf}.

\begin{proposition}\label{prop:schechtman zinn}
Let $n\in\N$ and $p\in[1,\infty)$. Suppose that $Z_1,\ldots,Z_n$ are independent $p$-generalized Gaussian random variables whose distribution has density
$$
f_p(x):= {1\over 2p^{1/p}\Gamma(1+{1\over p})}\,e^{-|x|^p/p}
$$
with respect to the Lebesgue measure on $\R$. Then, for $Z:=(Z_1,\ldots,Z_n)\in\R^n$, we have:
\vskip 1mm
\noindent(i) The random vector $Z/\|Z\|_p\in\SSS_p^{n-1}$ is independent of $\|Z\|_p$ and its distribution is $\mu_p$.

\noindent (ii) If $U$ is a random variable uniformly distributed on $[0,1]$ and independent of $Z$, then the random vector $U^{1/n} Z/\|Z\|_p$ is uniformly distributed on $\B_p^n$.
\end{proposition}

\subsection{Large deviations principles}

We start with the definition of a large deviations principle. In this subsection we denote for clarity the space dimension by $d$ instead of $n$ in order to distinguish it from our index parameter $n$. Finally, we make the assumption that all random objects we are dealing with are defined on a common probability space $(\Omega,\mathcal F,\Pro)$. For thorough introductions to the theory of large deviations, we refer the reader to the monographs \cite{DZ,dH} or the book \cite{Kallenberg}.

\begin{df}\label{def:ldp}
Let $\bX:=(X^{(n)})_{n\in\N}$ be a sequence of random vectors taking values in $\R^d$. Further, let $s:\N\to(0,\infty]$ be a positive sequence and $\mathcal{J}:\R^d\to[0,\infty]$ be a lower semi-continuous function. We say that $\bX$ satisfies a large deviations principle (LDP) with speed $s(n)$ and rate function $\mathcal{J}$ if
\begin{equation*}\label{eq:LDPdefinition}
\begin{split}
-\inf_{x\in A^\circ}\mathcal{J}(x) &\leq\liminf_{n\to\infty}{1\over s(n)}\log \Pro\big(X^{(n)}\in A\big)\\
&\leq\limsup_{n\to\infty}{1\over s(n)}\log \Pro\big(X^{(n)}\in A\big) \leq-\inf_{x\in\bar{A}}\mathcal{J}(x)
\end{split}
\end{equation*}
for all Borel sets $A\subseteq \R^d$. If $\mathcal J$ has compact level sets $\{x\in\R^d\,:\, \mathcal{J}(x) \leq \alpha \}$, $\alpha\in\R$, then $\mathcal J$ is called a good rate function. 
\end{df}

We notice that on the class of all $\mathcal{J}$-continuity sets, that is, on the class of Borel sets $A\subseteq \R^d$ for which $\mathcal{J}(A^\circ)=\mathcal{J}(\bar{A})$ with $\mathcal{J}(A):=\inf\{\mathcal{J}(x):x\in A\}$, one has the exact limit relation
$$
\lim_{n\to\infty}{1\over s(n)}\log \Pro\big(X^{(n)}\in A\big) = -\mathcal{J}(A)\,.
$$

Let $d\geq 1$ be a fixed integer and let $X$ be an $\R^d$-valued random vector. We write
$$
\Lambda(u)=\Lambda_X(u):=\log \E\, e^{\langle X,u\rangle}\,,\qquad u\in\R^d\,,
$$
for the cumulant generating function of $X$. Moreover, we define the (effective) domain of $\Lambda$ to be the set $D_{\Lambda}:=\{u\in\R^d:\Lambda(u)<\infty\}\subseteq\R^d$.

\begin{df}
The \textit{Legendre--Fenchel transform} of a convex function $\Lambda:\R^d \to(-\infty,+\infty]$ is defined as
$$
\Lambda^*(x):=\sup_{u\in\R^d}\big[\langle u, x\rangle -\Lambda(u)\big]\,,\qquad x\in\R^d\,.
$$
\end{df}

The Legendre--Fenchel transform of the cumulant generating function plays a crucial r\^ole in the following result, usually referred to as Cram\'er's theorem, (see, e.g., \cite[Theorem 2.2.30, Theorem 6.1.3, Corollary 6.1.6]{DZ} or \cite[Theorem 27.5]{Kallenberg}).

\begin{proposition}[Cram\'er's theorem]\label{prop:cramer}
Let $X,X_1,X_2,\ldots$ be independent and identically distributed random vectors taking values in $\R^d$. Assume that $0\in D_{\Lambda}^\circ$. Then the partial sums ${1\over n}\sum_{i=1}^n X_i$, $n\in\N$, satisfy an LDP with speed $n$ and good rate function $\Lambda^*$.
\end{proposition}

It will be important for us to deduce from an already existing large deviations principle a new one by applying a suitable transformation. The next result allows such a `transport' by means of a continuous function. This device is known as the contraction principle and we refer to \cite[Theorem 4.2.1]{DZ} or \cite[Theorem 27.11(i)]{Kallenberg}.

\begin{proposition}[Contraction principle]\label{prop:contraction principle}
Let $d_1,d_2\in\N$, and let $F:\R^{d_1}\to\R^{d_2}$ be a continuous function. Further, let $\bX=(X^{(n)})_{n\in\N}$ be a sequence of $\R^{d_1}$-valued random vectors that satisfies an LDP with speed $s(n)$ and the good rate function $\mathcal{J}_\bX$. Then the sequence $\bY:=(F(X^{(n)}))_{n\in\N}$ of $\R^{d_2}$-valued random vectors satisfies an LDP with the same speed and good rate function $\mathcal{J}_\bY=\mathcal{J}_\bX\circ F^{-1}$, i.e., $\mathcal{J}_\bY(y):=\inf\{\mathcal{J}_\bX(x):F(x)=y\}$, $y\in\R^{d_2}$, with the convention that $\mathcal{J}_\bY(y)=+\infty$ if $F^{-1}(\{y\})=\emptyset$.
\end{proposition}

\section{Arithmetic-geometric mean CLT for $p$-balls}

We shall now present the proof of Theorem \ref{thm:CLT}.

\begin{proof}[Proof of Theorem \ref{thm:CLT}]
Let $(X_1^{(n)}, \ldots, X_n^{(n)})$ be chosen uniformly at random from $\B_p^n$. Consider independent $p$-generalized Gaussians $Z_1,\dots,Z_n$ and a random variable $U$ uniformly distributed on $[0,1]$ and independent of the $Z_i$'s. We know from the Schechtman--Zinn result (Proposition \ref{prop:schechtman zinn}) that
\[
(X_1^{(n)}, \ldots, X_n^{(n)}) \stackrel{\text{d}}{=} U^{1/n} \frac{  (Z_1,\dots,Z_n)}{\|(Z_1,\dots,Z_n)\|_p}.
\]
With this representation, we have
\begin{equation} \label{eq:AGM equivalent}
\mathcal R_n = \frac{\Big(\prod\limits_{i=1}^{n}|X_i^{(n)}|\Big)^{1/n}}{ \Big(\frac1n \sum\limits_{i=1}^n {|X_i^{(n)}|}^p\Big)^{1/p}} \eqdistr
\frac{\Big(\prod\limits_{i=1}^{n}|Z_i|\Big)^{1/n}}{ \Big(\frac1n \sum\limits_{i=1}^n |Z_i|^p\Big)^{1/p}} = \frac{\exp \Big(\frac{1}{n}\sum_{i=1}^n \log|Z_i| \Big) }{ \Big(\frac1n \sum\limits_{i=1}^n |Z_i|^p\Big)^{1/p}},
\end{equation}
where \,$\eqdistr$\, denotes equality of distributions. If $(X_1^{(n)}, \ldots, X_n^{(n)})$ is chosen at random with respect to the cone measure $\mu_p$ on $\SSS_p^{n-1}$, then the factor $U^{1/n}$ does not appear in the Schechtman--Zinn representation (see Proposition \ref{prop:schechtman zinn} (i)) and the formula for $\mathcal R_n$ above does not change as the corresponding factor cancels out. For any $a\in\R$, the tail probability for $\mathcal R_n$ reads as follows,
\begin{align*}
& \Pro\big[\mathcal R_n \ge e^{m_p+a/\sqrt{n}}\,\big] \\
&\quad= \Pro\bigg[ \exp\bigg( -m_p -  \frac{a}{\sqrt{n}} + \frac1 n \sum_{i=1}^{n} \log |Z_i|  \bigg) \geq  \bigg(\frac1n \sum_{i=1}^n|Z_i|^p\bigg)^{1/p}\, \bigg] \\
&\quad= \Pro\bigg[ \exp\bigg( -a +  \frac{1}{\sqrt{n}} \sum_{i=1}^{n} (\log |Z_i| - m_p ) \bigg) \geq \bigg(1+ \frac1n \sum_{i=1}^n(|Z_i|^p - 1)  \bigg)^{\sqrt{n}/p}\, \bigg].
\end{align*}
Taking the logarithm, we can further write this  as
\begin{align} \label{eq:tail R}
&\mathrel{\phantom{=}} \Pro\big[\mathcal R_n \ge e^{m_p+a/\sqrt{n}}\,\big] \notag  \\
&=
\Pro\bigg[ \frac{1}{\sqrt{n} } \sum_{i=1}^{n}(\log |Z_i| - m_p) -  \frac{\sqrt{n}}{p}
\log \bigg(1+ \frac 1n \sum_{i=1}^n (|Z_i|^p - 1)\bigg) \geq a\bigg] \\
&=
\Pro\bigg[ \frac{1}{\sqrt{n} } \sum_{i=1}^{n}(\log |Z_i| - m_p)
-
\frac{1}{p\sqrt{n}} \sum_{i=1}^n (|Z_i|^p - 1) - \frac{\sqrt{n}}{p} \alpha\bigg(\frac 1n \sum_{i=1}^n (|Z_i|^p - 1)\bigg) \geq a\bigg], \notag
\end{align}
where the function $\alpha(x)$ is defined by $\log (1+x) = x + \alpha(x)$ for $x>-1$, so that  $\alpha(x) = o(x)$ as $x\to 0$.
Using Mathematica, we can find that $\E \log |Z_1|=m_p$, $\E |Z_1|^p=1$, and $\E |Z_1|^p \log |Z_1| = \frac1p(\log p + \psi(1+ \frac 1p) )=m_p + 1$, where in the last equality we used the property of the digamma function that $\psi(x+1) = \psi(x) +\frac1x$, $x>0$.  Let $(N_1, N_2)$ be a bivariate centered normal random vector with the same covariance matrix as that of $(\log|Z_1|, |Z_1|^p)$,  i.e.,\ with the diagonal elements $\Var(\log|Z_1|) = \frac{1}{p^2} \psi'(\frac1p)$ and $\Var(|Z_1|^p)=p$ and the covariance terms equal~$1$. Then the bivariate central limit theorem  states that
$$
\bigg( \frac{1}{\sqrt{n} } \sum_{i=1}^{n}(\log |Z_i| - m_p), \frac{1}{p\sqrt{n}} \sum_{i=1}^n (|Z_i|^p - 1)\bigg) \todistr (N_1, p^{-1}N_2),
$$
where $\todistrnoN$ denotes the distributional convergence. 
Further, by the strong law of large numbers, $\frac 1n \sum_{i=1}^n (|Z_i|^p - 1)$ converges to $0$ a.s. Using the relation $\alpha(x) = o(x)$ as $x\to 0$ together with  Slutsky's theorem, we arrive at
$$
\frac{\sqrt{n}}{p} \alpha\bigg(\frac 1n \sum_{i=1}^n (|Z_i|^p - 1)\bigg)
=
\frac 1{p\sqrt n} \sum_{i=1}^n (|Z_i|^p - 1) \cdot \frac{\alpha\bigg(\frac 1n \sum_{i=1}^n (|Z_i|^p - 1)\bigg)}{\frac 1n \sum_{i=1}^n (|Z_i|^p - 1)} \todistr 0
$$
since the first factor converges to $p^{-1}N_2$ in distribution, whereas the second one converges to $0$ a.s. Taking everything together and using the continuous mapping theorem together with Slutsky's theorem, we arrive at
$$
\frac{1}{\sqrt{n} } \sum_{i=1}^{n}(\log |Z_i| - m_p)
-
\frac{1}{p\sqrt{n}} \sum_{i=1}^n (|Z_i|^p - 1) - \frac{\sqrt{n}}{p} \alpha\bigg(\frac 1n \sum_{i=1}^n (|Z_i|^p - 1)\bigg)
\todistr
N_1 - p^{-1} N_2.
$$
Recalling \eqref{eq:tail R}, we obtain
\begin{equation}\label{eq:CLT_proof1}
\lim_{n \to \infty} \Pro\big[\mathcal R_n \ge e^{m_p+a/\sqrt{n}}\,\big]= \Pro \big[N_1 - p^{-1} N_2 \ge a\big] = 1-\Phi \left( \frac{pa}{\sqrt{\psi'(\frac1p)-p}} \right),
\end{equation}
where we used that $\Var(N_1 - p^{-1} N_2 )= \frac{1}{p^2} \psi'(\frac1p) - 2 \frac1p + \frac{1}{p} = \frac{1}{p^2} \psi'(\frac1p) -\frac{1}{p} $. 

To prove that $e^{m_p+a/\sqrt{n}}$ can be replaced with $e^{m_p} (1 + a/\sqrt{n})$ in the above equality~\eqref{eq:CLT_proof1}, fix some $\eps>0$ and note that $1+a/\sqrt n$ is sandwiched between $e^{(a-\eps)/\sqrt n}$ and  $e^{(a+\eps)/\sqrt n}$ provided $n$ is sufficiently large. Thus, for all such $n$,
$$
\Pro\big[\mathcal R_n \ge e^{m_p+(a+\eps)/\sqrt{n}}\,\big]
\leq
\Pro\big[\mathcal R_n \ge e^{m_p} (1+a/\sqrt{n})\,\big]
\leq
\Pro\big[\mathcal R_n \ge e^{m_p+(a-\eps)/\sqrt{n}}\,\big].
$$
Taking first the limit $n\to\infty$ (which is given by~\eqref{eq:CLT_proof1}) and then letting $\eps\downarrow 0$ and using the continuity of the function $\Phi$, we arrive at
$$
\lim_{n \to \infty} \Pro\bigg[\mathcal R_n \ge e^{m_p} \Big(1+\frac a{\sqrt{n}}\Big)\,\bigg]=
 1-\Phi \left( \frac{pa}{\sqrt{\psi'(\frac1p)-p}} \right).
$$
This proves the claim of Theorem~\ref{thm:CLT} for the uniform distribution on $\B_p^n$ and the cone probability measure $\mu_p$.

Consider now the case where $(X_1^{(n)},\dots, X_n^{(n)})$ is chosen with respect to the probability measure $\sigma_p$ on $\SSS_p^{n-1}$. It was proved in \cite[Theorem 2]{NR2003} that the total variation distance between $\mu_p$ and $\sigma_p$, we write $\dint_{\text{TV}}(\mu_p,\sigma_p)$, is bounded above by a constant $c_p\in(0,\infty)$ (only depending on $p$) times $n^{-1/2}$. Let us consider the sets
\[
A_n:= \left\{x=(x_1,\dots,x_n)\in\SSS_p^{n-1}\,:\, \frac{\Big(\prod\limits_{i=1}^{n}|x_i|\Big)^{1/n}}{ \Big(\frac1n \sum\limits_{i=1}^n {|x_i|}^p\Big)^{1/p}} \geq e^{m_p} \Big(1+\frac a{\sqrt{n}}\Big)\,  \right\},\qquad n\in\N.
\]
As we have shown  above,
\begin{align}\label{eq:central limit mu_p}
\lim_{n\to\infty}\mu_p(A_n) =  1 - \Phi\left( \frac{pa}{\sqrt{\psi'(\frac1p)-p}} \right).
\end{align}
Since by \cite[Theorem 2]{NR2003}, $|\sigma_p(A_n) - \mu_p(A_n)| \leq \dint_{\text{TV}}(\sigma_p,\mu_p) \stackrel{n\to\infty}{\longrightarrow} 0$,
we can replace $\mu_p$ by $\sigma_p$ in \eqref{eq:central limit mu_p} above.

We shall now briefly give the argument for Remark \ref{rem: reduced form}. Since this trivially holds when $\bX_n \in \SSS_p^{n-1}$, we consider only the case of the uniform distribution on the ball $\B_p^n$. Define the quantity
\[
\widetilde{\mathcal R}_n := n^{1/p} \Big(\prod_{i=1}^{n}|X_i^{(n)}|\Big)^{1/n} = U^{1/n}\cdot \mathcal R_n,
\]
where $U$ is uniformly distributed on $[0,1]$ and independent of $Z_1,\dots,Z_n$ as in Proposition \ref{prop:schechtman zinn}. Then
\[
\sqrt{n}\Big(\log\widetilde{\mathcal R}_n - m_p\Big) = \sqrt{n}\Big(\log\mathcal R_n - m_p\Big) + n^{-1/2}\log U.
\]
Since $n^{-1/2}\log U \to 0$ in probability, it follows from Slutsky's theorem that $\sqrt{n}(\log\widetilde{\mathcal R}_n - m_p)$ satisfies the same central limit theorem as $\sqrt{n}(\log\mathcal R_n - m_p)$. Hence, the analogue of~\eqref{eq:CLT_proof1} with $\mathcal R_n$ replaced by $\widetilde{\mathcal R}_n$ holds.
\end{proof}

\section{Arithmetic-geometric mean LDP for $p$-balls}

We shall present here the proof of Theorem \ref{thm:ldp}. Essentially, the result is a consequence of Cram\'er's theorem, the contraction principle and the probabilistic representation of Schechtman and Zinn. However, most of the work is a careful analysis required to obtain the rate function, which is represented in terms of the inverse function of $H(x) = \psi(x)-\log x$.

\begin{proof}[Proof of Theorem \ref{thm:ldp}]
We want to study the large deviations behaviour of the ratios $\mathcal R_n$ given in \eqref{eq:AGM equivalent}. As we shall see later, $\mathcal R_n$ can be written as a function of the partial sums
\[
\frac{1}{n}\sum_{i=1}^n \big(\log|Z_i|, |Z_i|^p\big)
\]
with i.i.d.\ increments $\big(\log|Z_i|, |Z_i|^p\big)$, $i\in\N$, and where the $Z_i$'s are i.i.d.\ $p$-generalized Gaussians. The goal is to prove a large deviations principle via Cram\'er's theorem and then to apply the contraction principle. In order to do that, we first need to check that $0\in D_\Lambda^\circ$. We have
\begin{align*}
\Lambda(s,t) & = \log \E e^{\langle (\log|Z_i|, |Z_i|^p),(s,t) \rangle} \cr
& = \log \left(\frac{1}{p^{1/p}\Gamma(1+\frac{1}{p})}\int_{0}^{\infty}e^{s\log x - (\frac{1}{p}-t)x^p}\,\dint x \right)\cr
& = \log \left( \frac{1}{p(1-pt)^{1/p}\Gamma(1+\frac 1p)}\Big(\frac{p}{1-pt}\Big)^{s/p} \Gamma\Big(\frac{s+1}{p}\Big) \right)\cr
& = - \frac{1}{p}\log(1-pt)+\frac{s}{p}\big[\log(p) - \log(1-pt)\big]+\log\Big(\Gamma\Big(\frac{s+1}{p}\Big)\Big) - \log \Gamma\left(\frac 1p\right),
\end{align*}
where $t<\frac{1}{p}$ and $s>-1$. Otherwise, we have $\Lambda(s,t)=+\infty$ as
can be seen from the second line. In particular, $D_\Lambda = (-1,+\infty)\times(-\infty,\frac{1}{p})$ and thus $0\in D_\Lambda^\circ$. This means that we can apply Cram\'er's theorem (Proposition \ref{prop:cramer}) and obtain an LDP for the partial sums
\begin{align}\label{eq: bivariate cramer sum}
\frac{1}{n}\sum_{i=1}^n \big(\log|Z_i|, |Z_i|^p\big)
\end{align}
with speed $n$ and the rate function given by the Legendre--Fenchel transform of $\Lambda$, which is
\begin{align*}
\Lambda^*(\alpha,\beta) & = \sup_{(s,t)\in\R^2}\Big[\big\langle (s,t),(\alpha,\beta) \big\rangle - \Lambda(s,t)\Big] \cr
& = \sup_{(s,t)\in(-1,\infty)\times (-\infty,\frac{1}{p})} \Big[\big\langle (s,t),(\alpha,\beta) \big\rangle - \Lambda(s,t)\Big].
\end{align*}
First, we observe that the function $(s,t)\mapsto \big\langle (s,t),(\alpha,\beta) \big\rangle - \Lambda(s,t)$ is concave on $(-1,\infty)\times (-\infty,\frac{1}{p})$. In the following we shall show that if $\beta>e^{p\alpha}$, then the gradient of this function vanishes at some unique point $(s^*,t^*)\in (-1,\infty)\times (-\infty,\frac{1}{p})$. This implies that the supremum is attained at this point.

Forming the partial derivatives to find the point $(s^*,t^*)$, we obtain the following two conditions:
\begin{align}\label{eq:alpha}
\frac{\partial}{\partial s}\Lambda(s^*,t^*) = \frac{1}{p}\bigg[ \frac{\Gamma'(\frac{s^*+1}{p})}{\Gamma(\frac{s^*+1}{p})}+\log\Big(\frac{p}{1-pt^*}\Big)\bigg] = \alpha
\end{align}
and
\begin{align}\label{eq:beta}
\frac{\partial}{\partial t}\Lambda(s^*,t^*) = \frac{s^*+1}{1-pt^*} = \beta.
\end{align}
Let $\beta>e^{p\alpha}$. To prove that the system \eqref{eq:alpha} and \eqref{eq:beta} has indeed a unique solution $(s^*,t^*)$, we define $v^*\in(0,+\infty)$ as the unique solution to the equation
\[
H(v^*) = p\alpha - \log(\beta),
\]
where we recall that $H(x) = \psi(x)-\log(x)$ is negative on $(0, \infty)$. In order to see that this equation  has a unique solution $v^*\in (0, \infty)$, we use the representation (see, e.g., \cite[Eq.\ 8.361.8, page 903]{GradshteynRyzhik})
$$
H(x) = \psi(x) - \log x = - \int_0^\infty e^{-tx} \Big( \frac{1}{1-e^{-t}} - \frac1t \Big) \dint t.
$$
The integrand is positive and decreases monotonously in $x>0$ for every $t>0$. Therefore, $H$ is monotone increasing, and it follows from the monotone convergence theorem that $H(0+)=-\infty$ and $H(+\infty)=0$ (for the former equality, we use that the integral diverges at $x=0$). Hence, $H$ is an increasing bijection between $(0,+\infty)$ and $(-\infty,0)$, which also shows that $m_p$ defined in Theorem \ref{thm:CLT} is negative.

Let $s^*>-1$ and $t^*<\frac{1}{p}$ be defined by  
\[
v^* = \frac{s^* + 1}{p}\qquad\text{and}\qquad t^* = \frac{1}{p} - \frac{s^*+1}{\beta p}.
\]
Then it can be easily checked that \eqref{eq:alpha} and \eqref{eq:beta} hold. The uniqueness follows from the fact that \eqref{eq:alpha} and \eqref{eq:beta} imply
\begin{align}\label{eq:H alpha beta}
H\Big(\frac{s^*+1}{p}\Big) = p\alpha - \log(\beta) <0,
\end{align}
which has a unique solution $s^*>-1$ as we have seen above. From this we can determine the unique solution $t^*$ to \eqref{eq:beta}.


The contraction principle (Proposition~\ref{prop:contraction principle}) shall be applied to the random vectors in \eqref{eq: bivariate cramer sum} with the continuous function
\[
F:(x,y)\mapsto \frac{e^x}{y^{1/p}}, \qquad x \in \R,\, y \ge e^{px}
\]
and $F(x,y)=1$ otherwise.
This means that the sequence of random variables
\[
F\bigg(\frac{1}{n}\sum_{i=1}^n \big(\log|Z_i|, |Z_i|^p\big) \bigg) = \frac{\Big(\prod\limits_{i=1}^{n}|Z_i|\Big)^{1/n}}{ \Big(\frac1n \sum\limits_{i=1}^n |Z_i|^p\Big)^{1/p}} \eqdistr \mathcal R_n
\]
satisfies an LDP with speed $n$ and a certain rate function $\mathcal J_p$ to be determined in the following. For $\theta \in (0,1)$, the rate function is given by
\[
\mathcal J_p(\theta) = \inf_{(\alpha,\beta):\,F(\alpha,\beta)=\theta} \Lambda^*(\alpha,\beta) = \inf_{\beta>0,\,\alpha=\log(\theta) +\frac{1}{p}\log(\beta)} \Big[\alpha s^* +\beta t^* - \Lambda(s^*,t^*)\Big],
\]
and $\mathcal J_p \equiv +\infty$  on $\R \backslash (0,1]$. Let $\theta\in(0,1)$. Note that \eqref{eq:beta} implies
\begin{equation} \label{eq: t^*}
\frac{\beta}{\frac{s^*+1}{p}} = \frac{p}{1-pt^*} \qquad\text{and}\qquad \frac{\beta}{p}-\frac{s^*+1}{p} = \beta t^*.
\end{equation} 
Now, using \eqref{eq: t^*} to exclude $t^*$ from $\Lambda(s^*,t^*)$, we obtain
\[
\Lambda(s^*,t^*) = \frac{s^*+1}{p}\log\Big(\frac{\beta}{s^*+1}\Big) + \frac{s^*}{p}\log p+\log\Big(\Gamma\Big(\frac{s^*+1}{p}\Big)\Big)- \log \Gamma\left(\frac 1p\right).
\]
Hence, excluding $\alpha$ and $\beta t^*$ using \eqref{eq: t^*} for the latter,
\begin{align} \label{eq: I, s^*, beta}
\mathcal J_p(\theta) & = \inf_{\beta>0} \bigg[s^* \log \theta +\frac{s^*}{p}\log \beta +\frac{\beta}{p} -\frac{s^*+1}{p} - \Lambda(s^*,t^*)\bigg] \cr
 & = \inf_{\beta>0} \bigg[s^* \log \theta +\frac{\beta}{p} -\frac{s^*+1}{p} -\frac{1}{p}\log \beta  \cr
 &\quad +\frac{s^*+1}{p}\log (s^*+1) - \frac{s^*}{p}\log p-\log\Big(\Gamma\Big(\frac{s^*+1}{p}\Big)\Big)  + \log \Gamma\left(\frac 1p\right) \bigg] .
\end{align}
Note that the equalities $F(\alpha,\beta)=\theta$ and  \eqref{eq:H alpha beta} imply
\begin{equation} \label{eq: s^*}
p\log \theta = H\Big(\frac{s^*+1}{p}\Big),
\end{equation}
Since $s^*$ given by  \eqref{eq: s^*} is a function of $\theta$ only and is independent of $\beta$ (under  $F(\alpha,\beta)=\theta$), it is clear that the infimum in \eqref{eq: I, s^*, beta} with $\theta \in (0,1)$ is attained at $\beta=1$ (which minimizes $p^{-1}(\beta - \log \beta)$). Thus,
\begin{align*}
\mathcal J_p(\theta)
&= \log(\theta)s^* + \frac{s^*+1}{p} \Big[ \log \Big(\frac{s^*+1}{p}\Big) -1 \Bigr] - \log\Big(\Gamma\Big(\frac{s^*+1}{p}\Big)\Big) \\
&\qquad\qquad+
\frac{1}{p} + \frac{1}{p}\log(p)  +\log \Gamma\left(\frac 1p\right).
\end{align*}
Finally, using \eqref{eq: s^*} to exclude $s^*$ and recalling that $G_p(\theta) = H^{-1}(p\log(\theta))$, we obtain
\begin{align*}
\mathcal J_p(\theta)  & = \big[p G_p(\theta) -1\big] \log(\theta) + G_p(\theta) \big[\log\big(G_p(\theta)\big) - 1\big] -\log \Gamma\big(G_p(\theta)\big) \\
&\qquad\qquad + \frac{1}{p} + \frac{1}{p}\log(p) + \log \Gamma\left(\frac 1p\right),\qquad \theta\in(0,1).
\end{align*}
We shall now prove that $\mathcal J_p(1)=+\infty$, where by the contraction principle
\[
\mathcal J_p(1) = \inf_{F(\alpha,\beta)=1} \Lambda^*(\alpha,\beta) = \inf_{(\alpha,\beta):\,\beta \leq e^{p\alpha}} \Lambda^*(\alpha,\beta).
\] 
We claim that for all pairs $(\alpha,\beta)$ with the property that $\beta\leq e^{p\alpha}$, we have
\[
\Lambda^*(\alpha,\beta) \equiv \sup_{s>-1,\,t<\frac{1}{p}} \big[\alpha s+\beta t - \Lambda(s,t) \big] = +\infty.
\]
To prove this it is enough to consider a sequence of pairs $(s_k,t_k)$ such that $s_k\to+\infty$ (as $k\to+\infty$) and 
\[
t_k:= \frac{1}{p} - \frac{s_k+1}{\beta p}, \qquad k\in\N,
\]
and to show that $\alpha s_k+\beta t_k - \Lambda(s_k,t_k)\to+\infty$ as $k\to+\infty$. It follows from the definition of $t_k$ that if $v_k= (s_k+1)/p$, then
\[
1-pt_k = \frac{v_k p}{\beta} \qquad\text{and}\qquad \frac{p}{1-pt_k} = \frac{\beta}{v_k}.
\]
Using the expression for $\Lambda(s_k,t_k)$ and excluding $s_k$ and $t_k$, we get
\[
\alpha s_k +\beta t_k - \Lambda(s_k,t_k) = v_k(p\alpha -\log(\beta)) + (v_k\log(v_k)-v_k) -\log(\Gamma(v_k)) + c(\alpha,\beta,p) 
\] 
where  $c(\alpha,\beta,p)$ is a term independent of the sequence $v_k$. Note that $v_k\to+\infty$ as $k\to+\infty$. Hence, by Stirling's formula, we have
\[
v_k\log(v_k) - v_k - \log(\Gamma(v_k)) = \frac{1}{2} \log(v_k) - \frac{1}{2}\log(2\pi)+o(1).
\]
Since $p\alpha - \log(\beta) \geq 0$, we have, as $k\to+\infty$,
\[
\alpha s_k +\beta t_k - \Lambda(s_k,t_k) \to +\infty.
\]
This proves that $\mathcal J_p(1)=+\infty$.

Therefore, we have proved an LDP for $\mathcal R_n$ with speed $n$ and rate function $\mathcal J_p$ as stated in Remark~\ref{rem}.
Since the function $H$ is continuous, so is $G_p$ on $(0,1)$, and consequently the same holds for $\mathcal J_p$ on the interval $(0,1)$. Thus, the LDP for $\mathcal R_n$ yields the limiting behavior presented in the statement of Theorem \ref{thm:ldp} in the case of the uniform distribution on $\B_p^n$ or the cone measure on $\SSS_p^{n-1}$. We will now prove that the same LDP holds for the uniform distribution on $\SSS_p^{n-1}$. Let $A\subseteq \R$ be a Borel set and define
\[
D_n:= \left\{x=(x_1,\dots,x_n)\in\SSS_p^{n-1}\,:\, \frac{\Big(\prod\limits_{i=1}^{n}|x_i|\Big)^{1/n}}{ \Big(\frac1n \sum\limits_{i=1}^n {|x_i|}^p\Big)^{1/p}} \in A  \right\},\qquad n\in\N.
\]
Then,
\[
\frac{1}{n} \log \sigma_p(D_n) = \frac{1}{n}\log \int_{D_n} h_{n,p}(x) \, \mu_p(\dint x),
\]
and it follows from Lemma \ref{lem:estimate density} that, as $n\to\infty$,
\[
\Big| \frac{1}{n} \log \mu_p(D_n) - \frac{1}{n}\log \sigma_p(D_n)\Big| \to 0.
\]
This shows the LDP for $\mathcal R_n$in the case of the surface probability measure $\sigma_p$ on $\SSS_p^{n-1}$ with the same rate function $\mathcal J_p$.

The limit relations $\mathcal J_p(0+) = \mathcal J_p(1-) = +\infty$ immediately follow from lower semi-continuity of $\mathcal J_p$ and the identities $\mathcal J_p(0) = \mathcal J_p(1) = +\infty$.

\end{proof}

\bibliographystyle{plain}
\bibliography{agm}

\begin{thebibliography}{10}

\bibitem{A2008}
J.M. Aldaz.
\newblock A refinement of the inequality between arithmetic and geometric
  means.
\newblock {\em J. Math. Inequal.}, 2(4):473--477, 2008.

\bibitem{A2010}
J.M. Aldaz.
\newblock Concentration of the ratio between the geometric and arithmetic
  means.
\newblock {\em J. Theoret. Probab.}, 23(2):498--508, 2010.

\bibitem{APT2018}
D.~Alonso-Guti\'errez, J.~Prochno, and C.~Th{\"a}le.
\newblock Large deviations for high-dimensional random projections of
  $\ell_p^n$-balls.
\newblock {\em ‎Adv. Appl. Math.}, 99:1--35, 2018.

\bibitem{DZ}
A.~Dembo and O.~Zeitouni.
\newblock {\em Large {D}eviations. {T}echniques and {A}pplications}, volume~38
  of {\em Stochastic Modelling and Applied Probability}.
\newblock Springer-Verlag, Berlin, 2010.
\newblock Corrected reprint of the second (1998) edition.

\bibitem{dH}
F.~den Hollander.
\newblock {\em Large {D}eviations}, volume~14 of {\em Fields Institute
  Monographs}.
\newblock American Mathematical Society, Providence, RI, 2000.

\bibitem{GKR}
N.~Gantert, S.S. Kim, and K.~Ramanan.
\newblock Large deviations for random projections of $\ell^{p}$ balls.
\newblock {\em Ann. Probab.}, 45(6B):4419--4476, 2017.

\bibitem{GM2003}
E.~Gluskin and V.D. Milman.
\newblock Note on the geometric-arithmetic mean inequality.
\newblock In {\em Geometric aspects of functional analysis}, volume 1807 of
  {\em Lecture Notes in Math.}, pages 130--135. Springer, Berlin, 2003.

\bibitem{GradshteynRyzhik}
I.S. Gradshteyn and I.M. Ryzhik.
\newblock {\em Table of integrals, series, and products}.
\newblock Elsevier/Academic Press, Amsterdam, seventh edition, 2007.

\bibitem{KPT17CCM}
Z.~Kabluchko, J.~Prochno, and C.~Th\"ale.
\newblock High-dimensional limit theorems for random vectors in
  $\ell_p^n$-balls.
\newblock {\em Commun. Contemp. Math. (online ready)}, 2017.

\bibitem{KPT2018}
Z.~{Kabluchko}, J.~{Prochno}, and C.~{Th\"ale}.
\newblock {Sanov-type large deviations in Schatten classes}.
\newblock {\em ArXiv e-prints}, August 2018.

\bibitem{Kallenberg}
O.~Kallenberg.
\newblock {\em Foundations of {M}odern {P}robability}.
\newblock Probability and its Applications. Springer-Verlag, New York, second
  edition, 2002.

\bibitem{Kim2017}
S.S. Kim.
\newblock Problems at the interface of probability and convex geometry: Random
  projections and constrained processes.
\newblock {\em Ph.D. thesis, Brown University}, 2017.

\bibitem{KimRamanan2015}
S.S. {Kim} and K.~{Ramanan}.
\newblock {A conditional limit theorem for high-dimensional $\ell^{p}$
  spheres}.
\newblock {\em Adv. Appl. Probab. (to appear)}.

\bibitem{NaorTAMS}
A.~Naor.
\newblock The surface measure and cone measure on the sphere of {$\ell_p^n$}.
\newblock {\em Trans. Amer. Math. Soc.}, 359(3):1045--1079 (electronic), 2007.

\bibitem{NR2003}
A.~Naor and D.~Romik.
\newblock Projecting the surface measure of the sphere of {$\ell_p^n$}.
\newblock {\em Ann. Inst. H. Poincar\'e Probab. Statist.}, 39(2):241--261,
  2003.

\bibitem{RachevRueschendorf}
S.T. Rachev and L.~R{\"u}schendorf.
\newblock Approximate independence of distributions on spheres and their
  stability properties.
\newblock {\em Ann. Probab.}, 19(3):1311--1337, 1991.

\bibitem{SchechtmanZinn}
G.~Schechtman and J.~Zinn.
\newblock On the volume of the intersection of two {$L^n_p$} balls.
\newblock {\em Proc. Amer. Math. Soc.}, 110(1):217--224, 1990.

\end{thebibliography}

\end{document}